\documentclass[preprint,12pt]{elsarticle}

\usepackage{amssymb}
\usepackage{amsmath}
\usepackage{amsthm}

\newcommand{\OO}{{\mathcal O}}
\newcommand{\MM}{{\mathcal M}}
\def\ad{\mathop{\rm ad}\nolimits}
\newcommand{\R}{\mathbf{R}}
\def\End{\mathop{\rm End}\nolimits}
\def\tr{\mathop{\rm tr}\nolimits}
\newcommand{\C}{\mathbf{C}}
\def\Hom{\mathop{\rm Hom}\nolimits}
\def\Real{\mathop{\rm Re}\nolimits}
\newcommand{\lie}[1]{\mathfrak{#1}}
\newtheorem{theorem}{Theorem}[section]
\newtheorem{proposition}[theorem]{Proposition}
\newtheorem{lemma}[theorem]{Lemma}

\newcommand{\PP}{{\rm P}}
\theoremstyle{remark}
\newtheorem{remark}{Remark}

\begin{document}

\begin{frontmatter}


\author{Nigel Hitchin}
 \ead{nigel.hitchin@maths.ox.ac.uk}
 \ead[url]{https://people.maths.ox.ac.uk/hitchin/}
 \affiliation{organization={University of Oxford},
           addressline={Mathematical Institute, Woodstock Road},
        city={Oxford},
            postcode={OX2 6GG},
           country={UK}}

\title{A universal Higgs bundle moduli space}

\begin{abstract}
The article  gives a differential geometric construction of the complex structure on the  total space of the family of moduli spaces of Higgs bundles on a curve $C$, as a fibration over Teichm\"uller space. The method uses a real valued function $f$ defined on the character variety, essentially the energy of a harmonic map, which is dependent on the complex structure of $C$. Using $f$ we define a natural Ehresmann connection with structure group the Hamiltonian diffeomorphisms of the fibre and using the associated horizontal distribution define an almost complex structure which we show to be integrable. The formalism parallels the usual variation of Hodge structure in this nonabelian version.

\end{abstract}




\end{frontmatter}



\section{Introduction}
\label{sec1}
The nonabelian Hodge correspondence establishes a  diffeomorphism between two  spaces associated to a compact Riemann surface $C$. One is the character variety, the space of reductive representations of the fundamental group $\pi_1(C)$ into a complex Lie group $G$, modulo conjugation. 
The other is the moduli space of stable Higgs bundles, pairs consisting of a holomorphic principal $G$-bundle $P$ and a holomorphic section $\Phi$ of $\ad (P)\otimes K$. The first,  denoted  usually  by $\MM_B$, the Betti moduli space,  is independent of the complex structure on $C$, but the complex structure of the second, the Dolbeault space $\MM_{Dol}$, varies. The purpose of this article is to give a description of this dependence by using a natural symplectic connection. 

The diffeomorphism between the two spaces preserves a real symplectic form. It is the real part of the Goldman form for $\MM_B$ and is the K\"ahler form  $\omega_1$ for the hyperk\"ahler metric on $\MM_{Dol}$. Our symplectic connection  may be interpreted as a connection over Teichm\"uller space ${\mathcal T}$ with structure group the Hamiltonian diffeomorphisms of $\MM_B$. It is defined by taking the trivial flat connection $\MM_B\times {\mathcal T}$ and considering the isometric action of the circle $\Phi\mapsto e^{i\theta}\Phi$ on  $\MM_{Dol}\cong \MM_B$. This provides a one-parameter family of flat connections and, averaging over the compact group $S^1$, the symplectic connection $A$. The family of flat connections then has the form 
$$\nabla_\theta=\nabla_A-\frac{i}{4}(e^{2i\theta}\varphi-e^{-2i\theta}\bar\varphi).$$
Considering its action on the adjoint bundle of Hamiltonian functions on $\MM_B$ endowed with the Poisson bracket, the vanishing of curvature gives the equations
$$
d_A\varphi=0, \qquad \{\varphi,\varphi\}=0,\qquad F_A+\frac{1}{8}\{\varphi,\bar\varphi\}=0.
$$
Here the curvature $F_A$ is to be regarded as a 2-form on Teichm\"uller space with values in functions on the fibres and $\varphi$ a 1-form with values in the complexification of this bundle. 

The form of these equations suggests a strong analogy with Higgs bundles, and if we observe in addition that the circle acts as $e^{2i\theta}$ on $\varphi$, it recalls the type of Higgs bundle known as a ``complex variation of Hodge structure". In fact, if one argues as Simpson did in  \cite{Simp}, that the differential geometric structure on $\MM_B$ is a nonlinear version of Hodge theory, then what we have here can indeed be considered as such a variation. 

The main application of this approach is to use the connection to define from a differential-geometric viewpoint the  complex structure on the total space $\MM_B\times {\mathcal T}\rightarrow {\mathcal T}$ which restricts on each fibre to  $\MM_{Dol}$ for the corresponding complex structure on $C$. From this point of view the product structure  $\MM_B\times {\mathcal T}$ is the isomonodromic foliation considered for example in the recent papers \cite{Col},\cite{Al}. The principal tool here is the function $f:{\MM_B}\rightarrow \R$ which, as observed in the author's original paper \cite{Hit0}, is at the same time the moment  map for the circle action on $\MM_{Dol}$, a K\"ahler potential for the complex structure on $\MM_B$, and essentially the energy of the harmonic map which lies at the basis of the nonabelian Hodge correspondence.

A number of observations follow from our approach including the behaviour of the energy for a fixed representation of $\pi_1(C)$ as the complex structure varies and the case of the character variety for a real form of the complex group $G$. We also describe hyperholomorphic line bundles on $\MM$ arising from deformations 
of the hyperk\"ahler metric.

The author wishes to thank ICMAT for its support:  a  lecture there by Richard Wentworth in June 2025 based on the work in progress for \cite{Col} was the stimulus for this paper and there are common results. That paper deals with the actual gauge-theoretical equations; here we use the geometry of a hyperk\"ahler  metric with a circle action. Another approach can be found in \cite{Al}.
\section{The function $f$}
\subsection{Basic properties}
Let $C$ be a compact Riemann surface of genus $g\ge 2$. A solution of the Higgs bundle equations for the general linear group (we focus on this case) consists of a holomorphic vector bundle $E$, a holomorphic section $\Phi$ of $\End E\otimes K$ and a Hermitian metric on $E$ defining a connection $A$ such that $F_A+[\Phi,\Phi^*]=0$. Equivalently 
$
\nabla_A+e^{i\theta}\Phi+e^{-i\theta}\Phi^*$
is a flat connection for all $\theta$. Following \cite{Hit0} we define 
$$f=-\frac{1}{2}\Vert \Phi \Vert^2=-i\int_C\tr \Phi\wedge\Phi^*.$$
The function $f$ is proper on the moduli space $\MM_{Dol}$ of solutions to these equations. At smooth points there is a natural hyperk\"ahler metric with complex structures $I,J,K$ satisfying the  quaternionic relations  $I^2=J^2=K^2=IJK=-1$ and corresponding K\"ahler forms $\omega_1,\omega_2,\omega_3$. We shall use $I$ for the holomorphic structure on $\MM_{Dol}$.  The action of the circle $\Phi\mapsto e^{i\theta}\Phi$ preserves $I$ and $\omega_1$ and $f$ is a corresponding moment map :  $i_X\omega_1=df$. 

The function $f$ has a different relationship with complex structure $J$ \cite{Hit0}. For a tangent vector $U$, $Jdf(U)=df(JU)=\omega_1(X,JU)=\omega_3(X,U)$ so that $Jdf=i_X\omega_3$. Then $dJdf=d(i_X\omega_3)={\mathcal L}_X\omega_3$. But the circle acts as $e^{i\theta}$ on $\omega_2+i\omega_3$ so we have 
\begin{equation}\label{dJdf}
\omega_2=-dJdf.
\end{equation}
By the nonabelian Hodge correspondence the complex structure $J$ defines the holomorphic structure on the moduli space of flat irreducible $GL(n,\C)$-connections  $\MM_B$ with representative $\nabla_A+\Phi+\Phi^*$.  Up to equivalence a flat connection is determined by its holonomy and so $\MM_B$ as a complex manifold can be identified with an open set in  $\Hom(\pi_1(C),GL(n,\C))/GL(n,\C)$ where the action is conjugation, the character variety. This depends only on the fundamental group and not the complex structure on $C$. Its natural Goldman-Atiyah-Bott holomorphic symplectic form has real and imaginary parts $\omega_1$ and $\omega_3$. So $J, \omega_3,\omega_1$ may be regarded as fixed. Then from (\ref{dJdf}) the function $f$ determines the hyperk\"ahler metric, since it defines $\omega_2$. The complex structures  $I$ and $K$ are obtained from $\omega_3^{-1}\omega_2$ and $\omega_1^{-1}\omega_2$, and so the quaternionic properties require $f$ to satisfy a nonlinear algebraic relationship among its second derivatives.  

\subsection{Variation of $f$}
The function $f$ depends on the complex structure $I$ which itself depends on the complex structure on $C$ and on the cotangent bundle this is just the Hodge star operator $\ast\!:\!T^*\rightarrow T^*$, with $\ast^2=-1$. A first order deformation is then an endomorphism $\dot \ast$ with $\dot \ast \ast+\ast \dot\ast=0$. We write the flat complex connection as $\nabla_A+\phi$ where $\phi$ is self-adjoint and 
$$f=-\frac{1}{2}\int_{C}\tr \phi\wedge \ast \phi.$$
We want to vary $f$ while keeping the point in the character variety fixed, so a first order deformation of the flat connection must be given by an infinitesimal  complex gauge transformation:
$$\nabla_A(\psi_1+\psi_2)+[\phi,\psi_1+\psi_2]$$
 where $\psi_1$ is skew-Hermitian and $\psi_2$ Hermitian. 
Since $\phi$ is Hermitian and the connection $A$ is unitary   we have a deformation $(\dot A,\dot \phi)$ where
$\dot A=\nabla_A \psi_1+[\phi,\psi_2]$ and 
\begin{equation}\label{dotphi}
 \dot\phi=\nabla_A \psi_2+[\phi,\psi_1].
 \end{equation}
Then 
$$\dot f=-\frac{1}{2}\int_{C}\tr \dot \phi\wedge \ast \phi+\tr \phi\wedge \dot\ast \phi+\tr  \phi\wedge \ast \dot \phi$$
Consider the first term and use (\ref{dotphi}). We have 
$$-\frac{1}{2}\int_{C}\tr \dot \phi\wedge \ast \phi=-\frac{1}{2}\int_{C}\tr ([\phi,\psi_1]\wedge \ast \phi)$$
using the relation $d_A\ast\phi=0$ from the Higgs bundle equations and integrating by parts. But $\alpha\wedge \ast \beta$ is symmetric and $\tr([a,b],c)$ skew-symmetric so this contributes $0$.
Similarly using $d_A\phi=0$  the third term gives zero and there remains
$$\dot f=-\frac{1}{2}\int_{C}\tr \phi\wedge \dot\ast \phi.$$
The variation $\dot\ast$ anticommutes with $\ast$ and so maps $(1,0)$ forms to $(0,1)$ forms. As such it is a section $\mu$ of $K^*\otimes \bar K$ (representing the Kodaira-Spencer class of the deformation of $C$) and 
\begin{equation}\label{dotf}
\dot f=-\frac{1}{2}{\Real}\int_{C}\tr \Phi^2\mu.
\end{equation}
Note that $\tr \Phi^2$ is the evaluation of an invariant polynomial on $\lie{gl}(n,\C)$ on the Higgs field $\Phi$ and so $\dot f$ is the real part of one of the Poisson-commuting functions of the integrable system on the moduli space $\MM_{Dol}$.

\subsection{Variation of the hyperk\"ahler metric} \label{HKvary}
Since we are fixing the character variety we regard the symplectic forms $\omega_1$ and $\omega_3$ as fixed so the first variation is simply $\dot \omega_2$.

\begin{proposition}
$\dot\omega_2$ is of  type $(1,1)$ with respect to all complex structures.
\end{proposition}

\begin{proof}
We write $\omega_{\zeta}=(\omega_2+i\omega_3)\zeta^2 +2i\omega_1\zeta+(\omega_2-i\omega_3)$, then regarding $\zeta$ as a parameter on $\PP^1\cong S^2$, this defines a multiple of the symplectic holomorphic 2-form for each complex structure of the family $x_1I+x_2J+x_3K$.  In particular $\zeta=\infty$ is the complex structure $I$ and  $(\omega_2+i\omega_3)$ is the holomorphic 2-form. 
The algebraic relations the symplectic forms satisfy are given by $\omega_{\zeta}^{n+1}=0$ for a manifold of complex dimension $2n$. Then
$(n+1)\omega_{\zeta}^n\dot\omega_\zeta=0$ for a  first order deformation $\dot\omega_{\zeta}$, which is equivalent to the vanishing of the $(0,2)$ component of $\dot\omega_{\zeta}$ for each complex structure $\zeta$. 

In our case $\dot\omega_{\zeta}=(\zeta^2+1)\dot \omega_2$ and so $\dot \omega_2$ has zero $(0,2)$ component but is also real so is of type $(1,1)$ with respect to all complex structures.

\end{proof}

The 2-form $\dot\omega_2=-d(Jd\dot f)$
 is clearly of type $(1,1)$ with respect to $J$ but we know from the Proposition  that it is also of type $(1,1)$ with respect to $I$, which we work with now. 
Consider $-2i(Jd\dot f)^{0,1}=(I-i)Jd\dot f=J(-I-i)d\dot f$ and recall from (\ref{dotf}) that $\dot f$ 
is the real part of an $I$-holomorphic function $h$ and so $(I+i)d\dot f=idh$ giving 
$$(Jd\dot f)^{0,1}=\frac{1}{2}Jdh.$$
Now $h$ is a Hamiltonian function of the integrable system and generates a holomorphic vector field $Z$ where $i_Z(\omega_2+i\omega_3)=dh$ and $IZ=iZ$. Then evaluating on a tangent vector $U$ we have
$$Jdh(U)=dh(JU) = (\omega_2+i\omega_3)(Z, JU)=g(JZ,JU)+ig(KZ,JU)$$
and this is $ g(Z,U)-ig(IZ,U)=2g(IZ,U)=2\omega_1(Z,U)$
and so 
\begin{equation}\label{Jdh}
Jdh=2i_Z\omega_1
\end{equation}

\begin{remark}

The last part of the above proof holds for any holomorphic function $h$ of the integrable system and so from (\ref{Jdh}) we have $dJdh$ is of type $(1,1)$ with respect to $J$ but also with respect to $I$ since $d\omega_1=0$ and $Z$ is holomorphic. The $(1,1)$-forms consist of  the $+1$ eigenspace of the action of $I$ on 2-forms so it follows from $IJ=K$ that $dJdh$ is of type $(1,1)$ with respect to all complex structures. This observation provokes the question of whether all  these first order deformations can be integrated to a genuine deformation.
\end{remark}

\begin{remark}
These 2-forms are exact and so each one may be considered as the curvature of  a hyperholomorphic connection on a $C^{\infty}$-trivial line bundle. In the context of mirror symmetry for $\MM_{Dol}$, a Fourier-Mukai type of transform is conjectured to  relate  a hyperholomorphic connection (a BBB-brane)  to a complex Lagrangian submanifold (a BAA-brane) of the mirror. The mirror for a group $G$ is,  by applying the SYZ approach,   the moduli space for the Langlands dual group in the sense that  the dual of the abelian variety which is a generic fibre of the integrable system for  one group is isomorphic to the moduli space of line bundles of some fixed degree on the mirror. From (\ref{Jdh}) we can define the holomorphic structure of a  hyperholomorphic line bundle by the $\bar\partial$-operator $\bar\partial+i_Z\omega_1$, defining a degree zero line bundle on each Jacobian fibre of the integrable system.

The Lagrangian for the {\it trivial} hyperholomorphic line bundle is recognized to be  the so-called Hitchin section, representing the character variety for the split real form in $\MM_B$. Then the non-trivial line bundle defined here corresponds to the application to this section of the time $t=1$ integral of the Hamiltonian vector field $Z$ to a holomorphic symplectic diffeomorphism of $\MM_{Dol}$. Clearly we can apply this to any other known correspondence of this type, for example the upward flow Lagrangians in \cite{HH}. Tensoring the  hyperholomorphic bundle by one of the above line bundles should be equivalent to translating the Lagrangian by the diffeomorphism. 
\end{remark}

\begin{remark}
 Identifying the $(1,1)$-form, a section of $T^*\otimes \bar T^*$, with $T\otimes \bar T^*$ using the holomorphic symplectic form gives the Kodaira-Spencer class of the deformation as an element of $H^1(\MM_{Dol},T)$. Then (\ref{Jdh})
shows that this is defined by taking  $i_Z\omega_1\in \Omega^{0,1}$ to represent a class in $H^1(\MM_{Dol},\OO)$ and applying the sheaf homomorphism  $\OO\rightarrow \OO(T)$ of taking the Hamiltonian vector field.  Note that a hyperholomorphic bundle corresponds to a holomorphic bundle on the twistor space $p: Z\rightarrow \PP^1$ and for a holomorphic line bundle  with trivial Chern class these are parametrized by $H^1(Z,\OO)$. Presumably restriction  to the fibre  of $p$ over the complex structure $I$ relates this to the Kodaira-Spencer class.
\end{remark}

This was something of a diversion but we shall need (\ref{Jdh}) later. We now address the geometry of the family of Higgs bundle moduli spaces over Teichm\"uller space and not just the first order variation.

\section{The symplectic connection} 

Consider the product  as a trivial bundle $\pi: \MM_B\times \mathcal T\rightarrow \mathcal T$ over Teichm\"uller space $\mathcal T$. For notational convenience we call the base $\mathcal T=B$ and refer to base and fibre derivatives. The product is a flat Ehresmann connection and differentiation in the direction of $B$ is the covariant derivative $\nabla_B$, acting on vertical vector fields, sections of the tangent bundle along the fibres $T_F$. A general Ehresmann connection is a horizontal subbundle of the tangent bundle of the total space and the covariant derivative of a vector field $V$ along the fibres by a vector field $Y$ in the base is the Lie bracket $[\tilde Y,V]$ of the horizontal lift $\tilde Y$ of $Y$.
 Since $\omega_1$ is constant on $\MM_B$, the connection also preserves Hamiltonian vector fields and acts on functions $h:\MM_B\times B\rightarrow \R$. Therefore $\nabla_Bh=d_Bh$ can also  be considered as  just the horizontal component of the derivative of $h$.
 
 The circle action acts on the fibres and so, applied to $\nabla_B$, gives a family of flat connections parametrized by the circle. 
Averaging over the circle (connections form an affine space) we get a connection $\nabla_A$ which is invariant and still preserves $\omega_1$ since this is also invariant by the circle action. We shall usually consider the covariant derivative of a function, so this is the adjoint action on a bundle of Lie algebras -- the real-valued functions on $\MM_B$ endowed with the Poisson bracket $\{h_1,h_2\}$ of $\omega_1$. Then 
$$   \nabla_B=\nabla_A+c$$
 where $c$ is a section of $\pi^*T^*_B$ and the average of  $c$ is zero. If we want to act on a vertical vector field or tensor we take the Hamiltonian vector field $X_c$ and the Lie derivative.
 
Let $X=X_f$ be the vertical vector field generating the circle action. 
Then, for a function $h$,  $\nabla_B\{f,h\}=\{\nabla_Bf,h\}+\{f,\nabla_Bh\}$ since $\nabla_B$ preserves the Poisson bracket. We rewrite this as 
$$\nabla_B{\mathcal L}_Xh-{\mathcal L}_X\nabla_Bh=\{\nabla_Bf,h\}.$$ 
But $\nabla_B=\nabla_A+c$ and by definition $\nabla_A$ is invariant, so   ${\mathcal L}_Xc=\nabla_Bf$. 

Now $\nabla_Bf$ is a section of $\pi^*T^*_B$ and given a tangent vector $Y$ at a point of $B$, $i_Y\nabla_Bf$ is just the variation of $f$ in the direction of $Y$ while fixing the point in $\MM_B$. But this is the $\dot f$ of the previous sections, and so 
$\nabla_Bf$ is the real part $\beta$ of 
\begin{equation}\label{phidef}
\varphi=\beta+i\gamma=-\frac{1}{2}\int_C \tr \Phi^2\mu
\end{equation}
It follows that ${\mathcal L}_Xc=\nabla_Bf=\beta$.

The circle action is $\Phi\mapsto e^{i\theta}\Phi$, so it  acts on $\varphi$ as  $e^{2i\theta}$. Then 
${\mathcal L}_X(\beta+i\gamma)=2i(\beta+i\gamma)$
giving  $\beta={\mathcal L}_X\gamma/2={\mathcal L}_Xc.$
Since  $c$ averages to zero by definition and $\varphi$ averages to zero since it transforms by $e^{2i\theta}$ we have  $c=\gamma/2$. We now have a formula for the symplectic connection:
\begin{equation}\label{conn}
\nabla_A=\nabla_B-\frac{1}{2}\gamma.
\end{equation}

The flat product connection is 
$$\nabla_B=\nabla_A+\frac{1}{2}\gamma=\nabla_A-\frac{i}{4}(\varphi-\bar\varphi)$$
so the famiy of flat connections parametrized by the circle is  
$$\nabla_\theta=\nabla_A-\frac{i}{4}(e^{2i\theta}\varphi-e^{-2i\theta}\bar\varphi).$$
and equating the Fourier components of the curvature  to zero gives
\begin{equation}\label{expand}
d_A\varphi=0, \qquad \{\varphi,\varphi\}=0,\qquad F_A+\frac{1}{8}\{\varphi,\bar\varphi\}=0
\end{equation} 
using $\omega_1$-Poisson brackets.

\begin{remark}
 The above relations look formally like the equations for a Higgs bundle, where the Lie algebra consists of the complexification of the real Lie algebra of Hamiltonian functions. 
 The brackets are however Poisson brackets and involve differentiation -- in terms of the Hamiltonian vector fields they generate this is the Fr\"olicher-Nijenhuis bracket on differential forms with values in vector fields. One might wonder therefore whether 
 derivatives in directions on the base are involved. Sections of $T\otimes T^*$ are linear combinations of terms  $\phi\otimes X,\psi\otimes Y$ and the formula for the bracket $[\phi\otimes X,\psi\otimes Y]$ is 
 $$\phi\wedge \psi\otimes [X,Y]+\phi\wedge {\mathcal L}_X\psi\otimes Y-{\mathcal L}_Y\phi\wedge \psi\otimes X
+d\phi\wedge i_X\psi\otimes Y+i_Y\phi\wedge d\psi\otimes X.$$
 In our case we have  vertical vector fields $X,Y$, so that $i_X\psi=0 =i_Y\phi$ and  the formula 
involves no  derivatives on the base.
\end{remark}

\begin{remark}
Since $\varphi$ evaluated on $\mu\in H^1(C,K^*)$ is essentially one of the Poisson commuting functions of the integrable system on the moduli space of Higgs bundles,  the relation $\{\varphi,\varphi\}=0$ suggests a relation with this, but we are taking the Poisson bracket with respect to the K\"ahler form $\omega_1$ and not the holomorphic form $\omega_2+i\omega_3$. In this case,   $\varphi$ is holomorphic and $\omega_1$ is of type $(1,1)$, so  if $g,h$ are holomorphic functions and   $i_Z\omega_1=dg$ then 
$Z=\sum _ia_i\partial/\partial \bar z_i$ annihilates $h$ and the Poisson bracket vanishes.
\end{remark}

\begin{remark}The analogy with Higgs bundles requires $\nabla_A$ to be the unitary connection and hence the real functions to be compared to the Lie algebra of the unitary group. But our  flat connections parametrized by the circle are then unitary connections, which is not the case for Higgs bundles. Put another way, the real functions are fixed by complex conjugation $h\mapsto \bar h$ whereas the Lie algebra of the unitary group is fixed by $A\mapsto -A^*$, so our expression $\{\varphi,\bar\varphi\}$ is analogous to $-[\Phi,\Phi^*]$. For finite-dimensional Lie groups  the equation $F_A-[\Phi,\Phi^*]=0$ is locally equivalent to a harmonic map from a surface  to the compact group $G$, as in \cite{Hit6}. 
\end{remark}

We now spell out the Higgs bundle-type relations (\ref{expand}):
\begin{itemize}
\item
$0=\{\varphi,\varphi\}=\{\beta+i\gamma,\beta+i\gamma\}=\{\beta,\beta\}-\{\gamma,\gamma\}+2i\{\beta,\gamma\}$
since the bracket on 1-forms is symmetric. Then $\{\beta,\beta\}=\{\gamma,\gamma\}$ and $\{\beta,\gamma\}=0$.
\item
$F_A=-\{\varphi,\bar\varphi\}/8=-\{\gamma,\gamma\}/4$ from the previous expansion.
\item
 $d_A\varphi=0$ or   equivalently 
$d_B\varphi =\{\gamma,\varphi\}/2$, so $d_B\beta=\{\gamma,\beta\}/2=0$, though since  $\beta=d_Bf, d_B\beta=0$ we know this already.
The imaginary part of the relation gives 
$d_B\gamma=\{\gamma,\gamma\}/2=-2F_A$.
\end{itemize}

\section{Applications}
\subsection{Holonomy}
By definition, the connection $\nabla_A$ is invariant by the circle action which means that $\nabla_Af=0$ and $\nabla_AX=0$.
Parallel translation along a curve $\gamma:[0,1] \rightarrow B$ with respect to the  connection consists of integrating a vector field on $\MM_B\times [0,1]$ to a Hamiltonian isotopy from the fibre over $t=0$ to the fibre over $t=1$, namely a diffeomorphism of $\MM_B$. Because $\MM_B$ is noncompact, in principle there is only a local solution to the differential equation,  but $f$ is preserved and  proper so given $R>0$ there is a solution on $\vert f\vert \le R$ for all $t\in [0,1]$. By uniqueness this extends to the whole of $\MM_B\times [0,1]$. 

A  consequence of the fact that parallel translation preserves the circle action   is that fixed points, or fixed points of subgroups of the circle, are Hamiltonian isotopic for any two points in Teichm\"uller space. An example is the moduli space of cyclic Higgs bundles -- fixed points of a finite subgroup of the circle. 
\subsection{The Levi form}
The base $B$ of the fibration, Teichm\"uller space, has a complex structure, its tangent space at $C$ isomorphic to $H^1(C,K^*)$. Given a point  $x\in \MM_B$, $f(x,t)$ is a real-valued function on the complex manifold $B$ and one may consider its Levi form $d_BI_Bd_Bf$. Recall the definition
$$\varphi=\beta+i\gamma=-\frac{1}{2}\int_C \tr \Phi^2\mu.$$
This is complex linear in $\mu$ and so a section of $\pi^*\Lambda^{1,0}T^*_B$ and therefore $I_B\varphi=i\varphi$ giving $I_B\beta=-\gamma$. We can therefore write the Levi form as $d_BI_Bd_Bf=d_BI_B\beta=-d_B\gamma$. But in the previous section we saw that $d_B\gamma=-2F_A$ so $2F_A$ is the Levi form, and so is of type $(1,1)$. 

However, the equation $F_A+\{\varphi,\bar\varphi\}/8=0$ gives the Levi form as $-\{\varphi,\bar\varphi\}/4$ and $\varphi$ is holomorphic in  the fibre directions, so this is $-\omega^{-1}_1(\partial_F\varphi,\partial_F\varphi)/4$ and  the  form is negative semi-definite. Since $\Vert \Phi\Vert^2=-2f$  equivalently the energy of the harmonic section of the flat $GL(n,\C)/U(n)$ bundle is plurisubharmonic. 

This is a variant of the result of D.Toledo in \cite{Tol}, but a recent paper by O.To\v sic \cite{Tos} goes further to analyse the null space of the Levi form. From our point of view the null space consists of the values of $\mu$ for which $\partial_F\varphi=0$, a critical point of a quadratic  function of the integrable system: this is Theorem 1.5 of \cite{Tos}. An alternative description (Theorem 1.1) describes the null space as the intersection of the horizontal spaces of the flat connection $\nabla_B$ and its transform by $i\in S^1$. In our formulation the two connections are $\nabla_A\pm \gamma/2$ so the 
kernel is defined by $X_{\gamma}=0$ or $d_F\gamma=0$, the critical locus again.

  \section{$\MM_B\times \mathcal T$ as a complex manifold}  \label{complex}
Consider the horizontal distribution $H$ on $\MM_B\times B$ defined by the connection $\nabla_A$. Recall that $\gamma$ is a section of $\pi^*T^*_B$ defining the action on functions; the action on sections of $T_F$ is given by the Lie derivative by the associated Hamiltonian vector field $X_{\gamma}$.  Then $X_{\gamma}$ is a section of $T_F\otimes \pi^*T^*_B$ and the horizontal lift of a tangent vector $Y$ on $B$ is 
$$\tilde Y =Y-\frac{1}{2}X_{\gamma(Y)}.$$
We introduce an almost complex structure on $T= T_F\oplus H$ by using $I$ on the tangent space along the fibres and the complex structure $I_B$ on Teichm\"uller space   using the isomorphism $d\pi: H\rightarrow \pi^*T_B$. 
\begin{proposition} The almost complex structure is integrable.
\end{proposition} 
\begin{proof} The $\bar\partial$-operator for this structure acting on functions is $\bar\partial_{\MM\times B}=\bar\partial_F + \bar\partial_A$ where the second term is the $(0,1)$ component of $\nabla_A$ and $\bar\partial_F$ is the operator along the fibres of $\pi:\MM_B\times B\rightarrow B$ using the complex structure $I$. Integrability is the equation $\bar\partial_{\MM_B\times B}^2=0\in \Omega^{0,2}(\MM_B\times B)$ which has three parts according to the decomposition of $T^*_{\MM_B\times B}$ into vertical and horizontal components.

Since $I$ is integrable in each fibre we have $\bar\partial_F^2=0$ and since the curvature $F_A$ has type $(1,1)$, $\bar\partial_A^2=0$. It remains to prove $\bar\partial_F\bar\partial_A+\bar\partial_A\bar\partial_F=0$. This is equivalent to saying that  $\bar\partial_A$ preserves local holomorphic functions on the fibres. Each fibre has a holomorphic symplectic form $\omega^c=\omega_2+i\omega_3$ and a function $h$ is holomorphic if $(\omega^c)^ndh=0$ where the complex dimension of $\MM_B$ is $2n$. The result will therefore hold from the following lemma:
\begin{lemma}
The section $\omega^c$ of $\Lambda^{2,0}T^*_F$  satisfies $\bar\partial_A\omega^c=0$.
\end{lemma}
\begin{proof}
First consider the variation of $\dot\omega^c$ of $\omega^c$ corresponding to a tangent vector $Y$ on $B$.
We have $\omega^c=\omega_2+i\omega_3$ and  $\omega_3$ is constant in our deformation so this is $\dot \omega_2$. In Section \ref{HKvary} we showed that
$$\dot\omega_2=-d_F(Jd_F\dot f)=-\frac{1}{2}d_FJ(d_Fh +d_F\bar h)$$
where $h=i_Y\varphi$, so \begin{equation}\label{first}
\bar\partial_B\omega_2=-d_FJd_F\bar \varphi/2
\end{equation}
The symplectic connection $\nabla_A =\nabla_B-\gamma/2$ acts on  sections of $\Lambda^*T^*_F$  by applying  the Lie derivative of  the Hamiltonian vector field $X_{\gamma}/2$. We have 
${\mathcal L}_{X_{\gamma}}(\omega_2+i\omega_3)=d_F(i_{X_{\gamma}}(\omega_2+i\omega_3))$ and, as in Section \ref{HKvary}, we consider a vertical tangent vector $U$ so that 
$$i_{X_{\gamma}}(\omega_2+i\omega_3)(U)=g(JX_{\gamma},U)+ig(KX_{\gamma},U)=\omega_1(X,JIU)+i\omega_1(X,KIU)$$
$$= d_F\gamma(J(I+i)U)=J(-I+i)d_F\gamma(U)=-Jd_F\bar \varphi(U)$$
and so, since $\nabla_A=\nabla_B- \gamma/2$, together with (\ref{first}) this proves the lemma.

\end{proof}

\end{proof}

A number of features follow from this construction, which is essentially an analytic version of Simpson's relative moduli space \cite{Simp}:
\begin{itemize}
\item
From $d_A\varphi=0$, the section $\varphi$ of $\pi^*T_B^*\subset T^*_{\MM}$ is a holomorphic 1-form on $\MM_{Dol}$.
\item
The inverse $\sigma$ of $\omega^c$ is a section of $\Lambda^2T_F\subset \Lambda^2T_{\MM}$ and is a holomorphic Poisson tensor, the symplectic leaves being the fibres of the holomorphic projection $\pi:{\MM}_B\times \mathcal T\rightarrow \mathcal T$.
\item
The $\C^*$-action along the fibres is holomorphic. 
\item
The product structure $\MM_B\times \mathcal T$ defines the isomonodromic foliation of the complex manifold. It is not holomorphic in general but it is of some interest (see \cite{Bis},\cite{Zuo} for example) to determine leaves which are. From the point of view here this is where the horizontal spaces for connections $\nabla_B$ and $\nabla_A$ coincide. 
\end{itemize}

\section{Examples}
\subsection{Genus $2$}
While we cannot produce explicit formulas for the function $f$ or the symplectic connection $\nabla_A$ we can in some cases demonstrate the holomorphic 1-form $\varphi$ on the universal Higgs bundle moduli space.

When the curve  $C$ has genus $g=2$ then the moduli space ${\mathcal N}$ of stable bundles of rank $2$ and fixed odd determinant is  isomorphic to the  intersection of two quadrics in $\PP^5$. The cotangent bundle $T^*{\mathcal N}$ is an open dense subset of the moduli space $\MM_{Dol}$ and, thanks to \cite{Betal} there is an explicit description of the integrable system which allows us to give an expression for $\varphi$.

The curve is hyperelliptic with equation 
$$y^2=(z-\mu_1)\dots (z-\mu_6)$$
and the Teichm\"uller space is a covering of  the configuration space of $6$ distinct points $\mu_i\in \PP^1$ modulo projective transformations. 
The moduli space ${\mathcal N}$ is the intersection of the standard quadric $Q: q(x)=\sum_ix_i^2=0$ with the variable quadric $Q_\mu: q_{\mu}(x)=\sum_i\mu_ix_i^2=0$. The integrable system (see also \cite{Hit5}), which is also valid in all dimensions, is defined by the functions: 
$$f_i=4\sum_{j\ne i}\frac{(x_iy_j-x_jy_i)^2}{\mu_j-\mu_i}.$$
There are linear relations to yield the three-dimensional space.

We can understand the formula by considering $x\wedge y\in \Lambda^2\C^6\cong \lie{so}(6)$, where $x$ is a null vector and $y$ with $(x,y)=0$ is non-null, as a coadjoint orbit of $SO(6)$, an open set in the cotangent bundle of the quadric $Q$. Restricting the cotangent bundle to  $Q\cap Q_\mu$,  the canonical symplectic form has a degeneracy foliation and the quotient gives the cotangent bundle of $Q\cap Q_{\mu}=\mathcal N$. The functions $x_iy_j-x_jy_i$ are linear coordinates on 
$\Lambda^2\C^6$ and one may check that the quadratic expression above is constant on the leaves of the foliation. 

The family of Higgs bundle moduli spaces can then be represented by this family of subvarieties of $T^*Q$ parametrized by $\{\mu_1,\mu_2,\mu_3,0,1,\infty\}$ and then 
\begin{equation}\label{phiformula}
\varphi=4\sum \sum_{j\ne i}\frac{(x_iy_j-x_jy_i)^2}{\mu_j-\mu_i}d\mu_i.
\end{equation}

\begin{remark}
The kernel of the Levi form   corresponds to the critical locus of the Hitchin fibration, and critical points imply that  the Higgs field vanishes at a point  \cite{Hit4}. From \cite{Hit5} this gives the equations
$$\sum_{i=1}^6\frac{x_i^2}{z-\mu_i}=0,\quad \sum_{i=1}^6\frac{y^2_i}{z-\mu_i}=0, \quad \sum_{i=1}^6\frac{x_iy_i}{z-\mu_i}=0$$
for some $z\in \PP^1$. Then for a generic point in the character variety, there is a complex structure for which the Levi form is degenerate.
\end{remark}
\subsection{Parabolic Higgs bundles on $\PP^1$}
A closely related situation concerns parabolic Higgs bundles on $\PP^1$ with marked points $\mu_1,\dots,\mu_n$, these points, up to projective equivalence, playing the role of the  base of the universal moduli space. Then a generic Higgs bundle is trivial with Higgs field 
$$\Phi=\sum_{i=1}^n\frac{A_i}{z-\mu_i}$$
where the $A_i$ are nilpotent matrices. We have 
$$\tr\Phi^2=\sum_{i,j}\frac{\tr (A_iA_j)}{(z-\mu_i)(z-\mu_j)}$$
and 
$$\varphi =\sum_{i\ne j}\frac{\tr (A_iA_j)}{(\mu_i-\mu_j)}d\mu_i.$$
\section{Further aspects}
\subsection{Real forms}
The moduli space of flat $G^r$-connections for a real form  $G^r\subset G$ is well-known to have a Higgs bundle interpretation: if $H\subset G^r$ is the maximal compact subgroup and $\lie{g}=\lie{h}\oplus \lie{m}$, then the principal $G$-bundle reduces to the complexification of $H$ and the Higgs field is a holomorphic 1-form with values in  the  bundle associated to the action on $\lie{m}$. In particular the circle action preserves this subspace of $\MM_{Dol}$ and the averaged connection induces one on this moduli space. In this case we only have the symplectic form $\omega_1$.

The obvious case is when $G^r$ is the maximal compact subgroup of $G$, so $U(n)\subset GL(n,\C)$ for  the case we have considered. The Higgs field vanishes here and so the two connections $\nabla_A,\nabla_B$ coincide. From the Narasimhan-Seshadri theorem, given a complex structure on $C$,  the moduli space of unitary connections is identified with the moduli space of stable bundles $\mathcal{N}$, and so the holomorphic structure on $\MM_B\times \mathcal {T}$ restricts to a corresponding universal space for stable bundles. 

\begin{remark}
Consider the formula for the representative of the Kodaira-Spencer class on $\MM_{Dol}$ in the direction $Y$: $\sigma\partial(i_Z\omega_1)\in \Omega^{0,1}(T)$ where $\sigma$ is the Poisson tensor given by the inverse of $\omega_2+i\omega_3$. 
Here $i_Z(\omega_2+i\omega_3)=dh$ where $h = \varphi(Y)$. 

Now $h$ is a quadratic function of $\Phi$ and so vanishes to second order on ${\mathcal N}$. There is therefore    a well-defined second derivative $\partial^2h$, a section of the symmetric power $S^2N^*$ of the conormal bundle of $\mathcal{N}\subset \MM_{Dol}$.
The symplectic form $\omega_2+i\omega_3$ defines an isomorphism of the conormal bundle with the tangent bundle of ${\mathcal N}$ since ${\mathcal N}$ is Lagrangian, so this second derivative is identified with a section $S$ of $S^2T_{\mathcal{N}}$. The Kodaira-Spencer class in $H^1({\mathcal N},T)$ induced by the one on $\MM_{Dol}$ is then  the contraction $\sum S^{ij}\omega_{j\bar k}$ of $S$ with $\omega_1$, as in \cite{Hit1}.
\end{remark} 
At the opposite extreme we may consider a component of $\Hom(\pi_1(C),SL(2,\R))$ realizing the uniformizing representations. Here the holomorphic vector bundle is $K^{1/2}\oplus K^{-1/2}$ and the Higgs field $\Phi(u,v)=(qv,u)$ for $q\in H^0(C,K^2)$. Then 
$$\varphi=-\frac{1}{2}\int_C \tr \Phi^2\mu=-\int_Cq\mu$$
and is an isomorphism from the moduli space to the cotangent space of $B$. Thus the universal family is isomorphic to  the cotangent bundle of Teichm\"uller space, with $\varphi$ now identified with the canonical 1-form. 

The other components for $SL(2,\R)$ \cite{Hit0} are of the form $E=L\oplus L^*$ and $\Phi(u,v)=(av,bu)$ for holomorphic sections $a,b$ of $L^2K,L^{-2}K$. Then $\varphi$ maps the universal family surjectively to the cotangent bundle of $\mathcal T$ but $b$ may vanish (at the positive-dimensional fixed point set of the circle) so this locus collapses to the zero section. The inverse image of a generic point involves the different partitions of the quadratic differential $ab$.

\subsection{The prequantum line bundle}
The symplectic form $\omega_1$ on ${\MM}_B$ defines  the curvature of a $U(1)$-connection $\nabla$ on a complex line bundle. In  the context of  geometric quantization  it is called the  prequantum line bundle $L$.

On $\MM_B\times B$  the form $\omega_1$ is  a section of $\Lambda^2T^*_F$  but in the direct sum decomposition  $T^*=H^*\oplus T^*_F$ of the connection $\nabla_A$ it can be promoted to a form $\tilde \omega_1$ on the product space. In the product decomposition $\MM_B\times B$  it of course extends trivially. 

The subbundle $T^*_F\subset T^*$ is characterized by the property that the interior product with a horizontal tangent vector is zero, and the same for $\Lambda^2T^*_F\subset \Lambda^2T^*$. Horizontal tangent vectors are of the form $Y-X_{\gamma(Y)}/2$ so 
 $$\tilde \omega_1=\omega_1-d_F\gamma/2.$$ From the definition of the complex structure in  Section \ref{complex}  in terms of vertical and horizontal subspaces this is of type $(1,1)$. 

Consider now
$$\omega_1-\frac{1}{2}d\gamma = \omega_1-\frac{1}{2}d_F\gamma-\frac{1}{2}d_B\gamma=\tilde\omega_1-\frac{1}{2}d_B\gamma.$$
Since $d_B\gamma=-2F_A$ and $F_A$ is of type $(1,1)$ on $B$ this is a closed 2-form of type $(1,1)$ on $\MM_B\times B$ and hence the curvature of a connection on a holomorphic line bundle. Moreover on each fibre of $\pi:\MM_B\times B\rightarrow B$ it is the prequantum line bundle. 

The connection $\nabla$ on $L$ over $\MM_B$, and not just its curvature $\omega_1$,  may be regarded as fixed,  then the covariant derivative  on sections of $L$ over $\MM_B\times B$ is  
$\nabla+\nabla_B-i\gamma/2 $
where $\gamma$ is acting by scalar multiplication. A local holomorphic section $s$ of $L$ then satisfies the equations:
\begin{equation}\label{holsections}
\nabla^{0,1}s=0,\qquad \nabla_B^{0,1}s-\frac{1}{2}(\nabla_{X_{\bar\varphi}}+i\bar\varphi)s=0.
\end{equation}

 In geometric quantization a  function $h$  acts on $C^{\infty}$ sections of the prequantum line bundle $L$ by 
$$
h\!\cdot\! s= \nabla_{X_h}s+ih s
$$
giving a representation of the Lie algebra of functions with respect to the Poisson bracket (here the real curvature form $\omega_1$ acts on the complex section $s$ by $-i\omega_1$).
We see from (\ref{holsections}) that the $\bar\partial$-operator of the connection $\nabla_A$ acts on sections of $L$ extended to $\MM_B\times B$  via the natural representation of the Lie algebra of functions. 
Since $\nabla_A$ is invariant by the circle action the  holomorphic structure preserves the finite-dimensional subspaces corresponding to weights of the action, the dimensions of which, the ``equivariant Verlinde formula" were calculated in \cite{And}.

\end{document}